\documentclass[12pt]{amsart}
\usepackage{amscd,amssymb,latexsym}
\usepackage{fullpage}
\newcommand{\Mdef}[2]{\newcommand{#1}{\relax \ifmmode #2 \else $#2$\fi}}

\newcommand{\sm }{\wedge}

\newcommand{\tensor}{\otimes}

\newcommand{\Hom}{\mathrm{Hom}}

\newcommand{\Ext}{\mathrm{Ext}}
\Mdef{\bhom}{\mathbf{\hat{H}om}}

\Mdef{\Mod}{\mathrm{mod}}

\newcommand{\st}{\; | \;}
\newcommand{\hash}{\#}

\newtheorem{thm}{Theorem}[section]
\newtheorem{lemma}[thm]{Lemma}

\newtheorem{cor}[thm]{Corollary}

\theoremstyle{definition}

\newcommand{\qqed}{\qed \\[1ex]}
\renewenvironment{proof}[1][\hspace*{-.8ex}]{\noindent {\bf Proof #1:\;}}{\qqed}

\Mdef{\PH} {\Phi^H}
\Mdef{\PK} {\Phi^K}
\Mdef{\PL} {\Phi^L}
\Mdef{\PT} {\Phi^{\T}}

\Mdef{\ef}{E{\cF}_+}
\Mdef{\etf}{\widetilde{E}{\cF}}
\Mdef{\eg}{E{G}_+}
\Mdef{\etg}{\tilde{E}{G}}

\Mdef{\infl}{\mathrm{inf}}
\Mdef{\defl}{\mathrm{def}}
\Mdef{\res}{\mathrm{res}}
\Mdef{\ind}{\mathrm{ind}}
\Mdef{\coind}{\mathrm{coind}}

\Mdef{\univ}{\mathcal{U}}

\Mdef{\Fp}{\mathbb{F}_p}
\Mdef{\Zpinfty}{\Z /p^{\infty}}
\Mdef{\Zpadic}{\Z_p^{\wedge}}

\newcommand{\lra}{\longrightarrow}

\newcommand{\spec}{\mathrm{Spec}}

\Mdef{\we}{\mathbf{we}}
\Mdef{\fib}{\mathbf{fib}}
\Mdef{\cof}{\mathbf{cof}}
\Mdef{\BI}{\mathcal{BI}}

\Mdef{\B}{\mathbb{B}}
\Mdef{\C}{\mathbb{C}}
\Mdef{\D}{\mathbb{D}}
\Mdef{\E}{\mathbb{E}}
\Mdef{\T}{\mathbb{T}}
\Mdef{\F}{\mathbb{F}}
\Mdef{\G}{\mathbb{G}}
\Mdef{\I}{\mathbb{I}}
\Mdef{\N}{\mathbb{N}}
\Mdef{\Q}{\mathbb{Q}}
\Mdef{\R}{\mathbb{R}}
\Mdef{\bbS}{\mathbb{S}}
\Mdef{\Z}{\mathbb{Z}}

\Mdef{\bA}{\mathbb{A}}
\Mdef{\bB}{\mathbb{B}}
\Mdef{\bC}{\mathbb{C}}
\Mdef{\bD}{\mathbb{D}}
\Mdef{\bE}{\mathbb{E}}
\Mdef{\bF}{\mathbb{F}}
\Mdef{\bG}{\mathbb{G}}
\Mdef{\bH}{\mathbb{H}}
\Mdef{\bI}{\mathbb{I}}
\Mdef{\bJ}{\mathbb{J}}
\Mdef{\bK}{\mathbb{K}}
\Mdef{\bL}{\mathbb{L}}
\Mdef{\bM}{\mathbb{M}}
\Mdef{\bN}{\mathbb{N}}
\Mdef{\bO}{\mathbb{O}}
\Mdef{\bP}{\mathbb{P}}
\Mdef{\bQ}{\mathbb{Q}}
\Mdef{\bR}{\mathbb{R}}
\Mdef{\bS}{\mathbb{S}}
\Mdef{\bT}{\mathbb{T}}
\Mdef{\bU}{\mathbb{U}}
\Mdef{\bV}{\mathbb{V}}
\Mdef{\bW}{\mathbb{W}}
\Mdef{\bX}{\mathbb{X}}
\Mdef{\bY}{\mathbb{Y}}
\Mdef{\bZ}{\mathbb{Z}}

\Mdef{\cA}{\mathcal{A}}
\Mdef{\cB}{\mathcal{B}}
\Mdef{\cC}{\mathcal{C}}
\Mdef{\mcD}{\mathcal{D}} 
\Mdef{\cE}{\mathcal{E}}
\Mdef{\cF}{\mathcal{F}}
\Mdef{\cG}{\mathcal{G}}
\Mdef{\mcH}{\mathcal{H}} 
\Mdef{\cI}{\mathcal{I}}
\Mdef{\cJ}{\mathcal{J}}
\Mdef{\cK}{\mathcal{K}}
\Mdef{\mcL}{\mathcal{L}}

\Mdef{\cM}{\mathcal{M}}
\Mdef{\cN}{\mathcal{N}}
\Mdef{\cO}{\mathcal{O}}
\Mdef{\cP}{\mathcal{P}}
\Mdef{\cQ}{\mathcal{Q}}
\Mdef{\mcR}{\mathcal{R}}
\Mdef{\cS}{\mathcal{S}}
\Mdef{\cT}{\mathcal{T}}
\Mdef{\cU}{\mathcal{U}}
\Mdef{\cV}{\mathcal{V}}
\Mdef{\cW}{\mathcal{W}}
\Mdef{\cX}{\mathcal{X}}
\Mdef{\cY}{\mathcal{Y}}
\Mdef{\cZ}{\mathcal{Z}}

\Mdef{\ca}{\mathcal{a}}
\Mdef{\ct}{\mathcal{t}}

\Mdef{\At}{\tilde{A}}
\Mdef{\Bt}{\tilde{B}}
\Mdef{\Ct}{\tilde{C}}
\Mdef{\Et}{\tilde{E}}
\Mdef{\Ht}{\tilde{H}}
\Mdef{\Kt}{\tilde{K}}
\Mdef{\Lt}{\tilde{L}}
\Mdef{\Mt}{\tilde{M}}
\Mdef{\Nt}{\tilde{N}}
\Mdef{\Pt}{\tilde{P}}

\Mdef{\tA}{\tilde{A}}
\Mdef{\tB}{\tilde{B}}
\Mdef{\tC}{\tilde{C}}
\Mdef{\tE}{\tilde{E}}
\Mdef{\tH}{\tilde{H}}
\Mdef{\tK}{\tilde{K}}
\Mdef{\tL}{\tilde{L}}
\Mdef{\tM}{\tilde{M}}
\Mdef{\tN}{\tilde{N}}
\Mdef{\tP}{\tilde{P}}

\Mdef{\ft}{\tilde{f}}
\Mdef{\xt}{\tilde{x}}
\Mdef{\yt}{\tilde{y}}

\Mdef{\Ab}{\overline{A}}
\Mdef{\Bb}{\overline{B}}
\Mdef{\Cb}{\overline{C}}
\Mdef{\Db}{\overline{D}}
\Mdef{\Eb}{\overline{E}}
\Mdef{\Fb}{\overline{F}}
\Mdef{\Gb}{\overline{G}}
\Mdef{\Hb}{\overline{H}}
\Mdef{\Ib}{\overline{I}}
\Mdef{\Jb}{\overline{J}}
\Mdef{\Kb}{\overline{K}}
\Mdef{\Lb}{\overline{L}}
\Mdef{\Mb}{\overline{M}}
\Mdef{\Nb}{\overline{N}}
\Mdef{\Ob}{\overline{O}}
\Mdef{\Pb}{\overline{P}}
\Mdef{\Qb}{\overline{Q}}
\Mdef{\Rb}{\overline{R}}
\Mdef{\Sb}{\overline{S}}
\Mdef{\Tb}{\overline{T}}
\Mdef{\Ub}{\overline{U}}
\Mdef{\Vb}{\overline{V}}
\Mdef{\Wb}{\overline{W}}
\Mdef{\Xb}{\overline{X}}
\Mdef{\Yb}{\overline{Y}}
\Mdef{\Zb}{\overline{Z}}

\Mdef{\db}{\overline{d}}
\Mdef{\hb}{\overline{h}}
\Mdef{\qb}{\overline{q}}
\Mdef{\rb}{\overline{r}}
\Mdef{\tb}{\overline{t}}
\Mdef{\ub}{\overline{u}}
\Mdef{\vb}{\overline{v}}

\Mdef{\hc}{\hat{c}}
\Mdef{\he}{\hat{e}}
\Mdef{\hf}{\hat{f}}
\Mdef{\hA}{\hat{A}}
\Mdef{\hH}{\hat{H}}
\Mdef{\hJ}{\hat{J}}
\Mdef{\hM}{\hat{M}}
\Mdef{\hP}{\hat{P}}
\Mdef{\hQ}{\hat{Q}}

\Mdef{\thetab}{\overline{\theta}}
\Mdef{\phib}{\overline{\phi}}

\Mdef{\uA}{\underline{A}}
\Mdef{\uB}{\underline{B}}
\Mdef{\uC}{\underline{C}}
\Mdef{\uD}{\underline{D}}

\Mdef{\bolda}{\mathbf{a}}
\Mdef{\boldb}{\mathbf{b}}
\Mdef{\bfD}{\mathbf{D}}

\Mdef{\fm}{\frak{m}}
\Mdef{\fp}{\frak{p}}

\newcommand{\fX}{\mathfrak{X}}

\Mdef{\eps}{\epsilon}

\input{xypic}
\newcommand{\Qt}{\widetilde{\Q}}
\newcommand{\Qc}{\Q [c]}
\newcommand{\Qd}{\Q [d]}
\renewcommand{\Lc}{\Q [c, c^{-1}]}
\newcommand{\II}{\mathbb{I}}
\newcommand{\cOcC}{\cO_{\cC}}
\newcommand{\cEi}{\cE^{-1}}
\newcommand{\cQh}{\cQ^{\hash}}
\newcommand{\cQp}{\mathcal{Q}'}
\newcommand{\cQph}{(\cQp)^{\hash}}
\newcommand{\mcDh}{\mcD^{\hash}}
\newcommand{\mcDp}{\mathcal{D}'}
\newcommand{\mcDph}{(\mcDp)^{\hash}}
\newcommand{\cCh}{\cC^{\hash}}
\newcommand{\cCt}{\widetilde{\cC}}
\newcommand{\cCth}{\cCt^{\hash}}
\newcommand{\shv}{\mathrm{Shv}}
\newcommand{\sub}{\mathrm{Sub}}
\newcommand{\Gammat}{\widetilde{\Gamma}}
\newcommand{\Gt}{\widetilde{G}}
\newcommand{\SGt}{\widetilde{SG}}
\newcommand{\modules}{\mbox{-mod}}
\newcommand{\etc}{\tilde{E}\cC}

\begin{document}
\title{The rational homotopy groups of virtual spheres for rank 1
  compact Lie groups}

\author{J.P.C.Greenlees}
\address{Mathematics Institute, Zeeman Building, Coventry CV4, 7AL, UK}
\email{john.greenlees@warwick.ac.uk}

\date{}

\begin{abstract}
  We calculate the rational representation-ring-graded stable stems
  for rank 1 groups, $SU(2), SO(3), Pin (2), O(2), Spin(2)$ and
  $SO(2)$, in the same spirit as the calculations for finite groups in
  \cite{GSerrefinite}. 
\end{abstract}

\thanks{   The author is grateful to J.D.Quigley for his interest,
  which encouraged the author to write this down. The author
    would also  like to thank the Isaac Newton   Institute for
  Mathematical Sciences, Cambridge, for support and   hospitality
  during the programme Equivariant Homotopy Theory in  Context, where
  this paper was written. This work was supported by EPSRC grant EP/Z000580/1.  } 
\maketitle

\tableofcontents

\section{Context}

\subsection{Precursors}
In stable homotopy theory, the spheres $S^n$ are basic building
blocks, and also invertible. Serre's
Finiteness Theorem states that they are rationally very simple in the
sense that the group $[S^m,S^n]$ of stable maps is finite if $m\neq
n$; since $[S^n,S^n]=\Z$, this means that rationally the structure of stable
homotopy is very simple. Equivariantly, rationalisation is still a
massive simplification, but
the residual structure in the rationalisation is more interesting. 

If $G$ is a compact Lie group and  $V, W$ are finite dimensional orthogonal real
representations of $G$, it is natural to consider their one-point
compactifications $S^V, S^W$, and then ask to calculate the group $[S^V,
S^W]^G$ of stable equivariant maps.  If $G$ is finite these are again
finitely generated, and tom Dieck splitting reduces the computation to
a non-equivariant one when $V$ is a summand of $W$. For general $V,W$,
a rational computation was given in
\cite{GSerrefinite} in terms of the rational representation theory of
subquotients of $G$. The basic calculation of the rational $\Z$-graded maps
$$[S^V, S^W]^G_*\tensor \Q=\prod_{(H)}\Hom_{W_G(H)}(\pi_*(S^{V^H}),
\pi_*(S^{W^H})), $$
lets one understand the rank of $[S^V,S^W]^G$ in terms of fixed point
representations. 

However if $G$ is not finite, the groups $[S^V,S^W]^G$ are often not
finitely generated. For example, tom Dieck calculated the
Burnside ring $A(G)=[S^0,S^0]^G$:
taking $f:S^0\lra S^0$ to the function $H\longmapsto
\deg (f^H)$ gives an isomorphism with the ring of continuous, integer valued
functions on the space $\Phi G$ of subgroups of $G$ with  finite Weyl group:
$$[S^0,S^0]^G\stackrel{\cong}\lra C(\Phi G,\Z). $$
 As an abelian group, this is finitely generated if and only if
the identity component of $G$ is a central torus. Despite this, we can get
a reasonable understanding of $[S^V,S^W]^G$ by rationalizing, and it
is the purpose of this paper to give a complete answer when $G$ is a
rank 1 group. We will emphasize the overall shape of the answer, and
particularly the fact that virtual representations can behave rather
differently to actual representations; this is not apparent for finite
groups, but becomes more and more apparent as the rank increases.

There are now algebraic models for rational equivariant spectra for all 
compact Lie groups of rank 1. The models are designed for
calculation,  as we illustrate here. As in the case
of finite groups \cite{GSerrefinite},  the calculation is
entirely reduced to real representation theory.

The present note should give enough detail for the reader to follow
the calculations once the algebraic categories are accepted as models
of $G$-spectra, and some may view this as an entry point. 
References for proofs are given along the way, and some may find those
longer accounts  a better introduction. A broader context is given in the
survey \cite{AGconj}. The book \cite{BtD} is used as a reference for
representation theory of compact Lie groups.

\subsection{The groups}
We  will concentrate on the 6 most obvious rank 1 compact Lie groups. 
However, in the light of \cite{gq1} this can easily be adapted to any
rank 1 group.  The two semisimple rank 1 groups are  $SU(2)$ and 
$SO(3)$. For each of those we consider also the maximal torus and its normalizer.

Altogether we adopt the following abbreviations 
$$\xymatrix{
  SU(2)\rto & SO(3)&&\Gammat\rto &\Gamma\\
  Pin (2)\rto \uto& O(2)\uto&&\Gt\rto \uto & G\uto\\
  Spin (2)\rto \uto& SO(2)\uto&&\SGt \rto \uto & SG\uto
}$$
In some sense $G=O(2)$ is the core case, which helps explain the
naming scheme. In all cases the double cover is indicated by a tilde.
Although  $Spin (2) $ and $SO(2)$ are both copies of the circle group,
one needs to remember the quotient map between them is of degree 2 and
does not induce a bijection on subgroups.

\subsection{Conventions}
By stability we have $[S^V,S^W]^G=[S^0, S^{W-V}]^G$, and we will 
generally consider the virtual representation $U=W-V$. Since integer 
spheres shift homotopy groups in a simple fashion, it is convenient
 to treat fixed spheres differently: we assume $U^G=W^G-V^G=0$ and then calculate 
$[S^0, S^U]^G_*$ (integer gradings). 

We write $\Gamma=SO(3), G=O(2), SG=SO(2)$ and $W=G/SG$;
double covers will be indicated by a tilde. 

Without further explicit mention,  all spectra are rational.

\subsection{Contents}
The purpose of this paper is to calculate $[S^0,S^U]^H_*$ when 
$H$ is one of the 6 rank 1 groups we named  above, and where $U^H=0$. 

The key input is for the circle group, so we begin with that in
Section \ref{sec:so2}. Building on this, we move to $O(2)$, with
general notation in Section \ref{sec:o2} and then the cyclic and
dihedral parts in the following two sections. The transition to the group
$SO(3)$ only involves noting fusion and the change in representation
theory: this is dealt with in Section \ref{sec:so3}. Finally we turn to the
double covers, which are minor variants. We deal with $Pin(2)$ in
Section \ref{sec:pin2} and with $SU(2)$ in Section \ref{sec:su2}.

\section{The group $SO(2)$}
\label{sec:so2}
In this section the ambient group is $SG=SO(2)$. It is reasonable to
do this first, since it is the ingredient with the most substance.
All the calculations in this section can be made with \cite{s1q}, since 
the Adams spectral sequence collapses to a short exact sequence.


\subsection{Subgroups}
The set of finite subgroups of $SG$ is $\cC=\{ C_1, C_2, \ldots \}$,
and we will use the letter $s$ to index these groups $C_s$.
We write $\fX_{SG}=\sub(SG)/SG$  for the space
 of conjugacy classes of subgroups.  With the Hausdorff metric
 topology, it is the one point compactification $\cCh$, with $SG$ as the
compactifying point. The Zariski topology will not play an explicit
role here. 

\subsection{Representations}
The trivial representation is real, but all other representations
$z^n$ are complex, so we write $U=a+\bigoplus_{n\neq 0}a_n z^n$ for
integers $a, a_n$, almost all zero. In our case $a=0$ because of our assumption that $U^{SG}=0$.

We just need to remember the dimension function
$d_U: \cC \lra \Z$. Thus
$$d_U(s)=\dim_{\R}(U^{C_s})=a+2\sum_{s|n} a_n. $$
The function $d_U$  takes finitely many values, so we may let 
 $m_U=\min \{ d_U(s)\st s\geq 1\},
M_U=\max \{ d_U(s)\st s\geq 1\}$ and call $[m_U,M_U]$ the {\em
  range} of $U$.
\subsection{Notation}
We summarize notation for essential ingredients and conventions. These
will be constant companions throughout the calculation.
\begin{itemize}
\item $\Qc=H^*(BSG)$ where $c\in H^2(BSG)$ is of degree $-2$.
\item $\Lc$, the Laurent ring.  
\item Define $I$ by the short exact sequence, 
$0\lra \Qc\lra \Lc\lra I\lra 0$.
\item  There is an isomorphism  $I\cong \Sigma^2\Qc^*, $
where $V^*=\Hom_{\Q}(V,\Q)$ denotes graded vector space duality. Note the
double suspension, and remember it!
\item $\cOcC=\prod_s \Qc$, the ring of functions on $\cC$.
\item $\cE=\langle c^v\st v:\cC\lra \N \mbox{ zero a.e. }\rangle
  \subseteq \cOcC$, the
  multiplicatively closed set of Euler classes.
\item $t=\cEi \cOcC$, the $\cC$-Tate ring for the sphere spectrum
  \item Define $\II$ by the short exact sequence 
$0\lra \cOcC\lra t\lra \II\lra 0$.
\item There is an isomorphism 
$\II\cong \bigoplus_s I_s$, where $I_s$ is the copy of $I$ for the   $s$th factor $\Q[c]$. 

  \end{itemize}

  \subsection{The standard model}
 The standard model $\cA(SG)$ is an abelian category  which provides
 an algebraic model for rational $SG$-spectra in the sense that there
 is a Quillen equivalence \cite{tnqcore}:
  $$  \mbox{$SG$-spectra}\simeq \mbox{DG-$\cA(SG)$}. $$
We summarize necessary facts about $\cA (SG)$ from 
\cite{s1q}.

The objects of $\cA(SG)$ are 
$(N\lra t\tensor V)$ where $N$ is a $\cOcC$-module, $V$ is a graded
rational vector space and the map $N\lra t\tensor V$ is inverting $\cE$. This should
be thought of as $N$ (the `nub') together with the rigidification that
the vector space $V$ gives a `basis' of $\cEi N$. Morphisms $(N\lra
t\tensor V)\lra (N'\lra t\tensor V')$ are given by compatible pairs
$\theta : N \lra N'$ and $\phi : V\lra V'$. The object associated to
an $SG$-spectrum $X$ is
$$\pi^{\cA}_*(X)=(\pi^{SG}_*(DE\cC_+\sm X) \lra \pi^{SG}_*(DE\cC_+\sm
\etc \sm X)) $$
so
$$ N= \pi^{SG}_*(DE\cC_+\sm X) \mbox{ and } V=\pi_*(\Phi^{SG}X), $$
where $D(\cdot )=F( ., S^0)$ is functional duality and $\Phi^{SG}$ is
geometric fixed points; $\cOcC=\pi^{SG}_*(DE\cC_+)$ and
$t=\pi^{SG}_*(DE\cC_+\sm \etc)$. 

The category $\cA (G)$  has injective dimension 1 and there is a short exact sequence
$$0\lra \Ext(\Sigma H_*(X), H_*(Y))\lra [X,Y]^{SG}\lra \Hom (H_*(X),
H_*(Y))\lra 0.$$
It follows from this short exact sequence 
that all objects  are formal (i.e., equivalent to an object with 
  zero differential), and we may therefore work with objects of
  $\cA(SG)$. We will often use notation for objects of $\cA (SG)$  that is imported 
  from $SG$-spectra.

The category $\cA (SG)$ is  symmetric monoidal with tensor unit
$$S^0=(\cOcC\lra t\tensor \Q).$$

One checks that if $T$ is a torsion $\cOcC$-module in the sense that
$\cEi T=0$ then $T=\bigoplus_s T_s$ and $T_s$ is a torsion
$\Qc$-module for all $s$. We then define
$$\Sigma^UT=\bigoplus_s \Sigma^{d_U(s)}T_s. $$
More generally we may define $\Sigma^UN$ for any $\cOcC$-module. This
relies on the fact that $d_U$ is zero almost everywhere. Thus if
$\cC_d$ denotes the subgroups on which $d_U$ takes the value $d$,
$\cC_0$ is cofinite and the partition is finite.  We may define
$$\Sigma^UN=\prod_d \Sigma^{d}e_{\cC_d}N. $$

\subsection{The calculation}
The objects $e(V)=(t\tensor V\lra t\tensor V)$, for a graded vector
space $V$, and $f(T)=(T\lra 0)$, for a torsion $\cOcC$-module $T$,  
are useful. Indeed, if $X=(M\lra t\tensor U)$ then 
$$\Hom (X, e(V))=\Hom (U,V), \Hom (X, f(T))=\Hom (M,T)$$
Thus $e(V)$ is injective for any $V$ and if $T$ is a torsion injective
$\cOcC$-module, $f(T)$ is
injective. Thus we have an injective resolution
$$0\lra S^0 \lra e(\Q)\stackrel{i}\lra f(\II)\lra 0. $$

By suspension, we have $S^U=(\Sigma^U\cOcC \lra t\tensor \Q)$ and an injective
resolution
$$0\lra S^U \lra e(\Q)\stackrel{i}\lra f(\Sigma^U \II)\lra 0. $$
The Ext groups $\Ext^* (S^0, S^U)$ are the homology of the cochain complex
$$\xymatrix{
  \Hom (S^0,e(\Q))\rto^{i_*} \ar@{=}[d]&\Hom (S^0,f(\Sigma^U\II ))\ar@{=}[d]\\
  \Hom (\Q,\Q)\rto \ar@{=}[d]&\Hom (\cOcC ,\Sigma^U \II )\ar@{=}[d]\\
  \Q\rto &\Sigma^{U}\II
  }$$
in cohomological degrees 0 and 1.   We may read off the answer.
  
  \begin{lemma}
 The vector space $(\Sigma^{U}\II)_0$ is finite dimensional: it  is a
 sum with  one copy of $\Q$ for each $s$ with $d_U(s)<0$. 
 The connecting map $\Q \lra \Sigma^{U}\II$ is the diagonal. In
 particular it is zero if $d_U(s)\geq 0$ for all $s$ and it is nonzero
 if there is some $s$ with $d_U(s)<0$.
\end{lemma}

\begin{cor}
  In even degrees $[S^0,S^U]^{SG}_*$ is $\Q$ if $d_U\geq 0$ and 
it is 0 if  $d_U$ takes a negative value. 

In odd degrees $[S^0,S^U]^{SG}_*$ is $\Sigma^{U-1} \II$ if $d_U\geq 
0$ and it is the quotient of this by the diagonal $\Delta_{-1} \cong \Q$ in degree $-1$ if 
$d_U$ takes a negative value. 
\end{cor}

Thus the even degree homotopy is of dimension 0 or 1. The homotopy
in odd degrees $\geq 1$ is a vector space of countable dimension. It
may be non-zero and finite dimensional in degrees down to $m_U+1$ and
in negative degrees less than this, the homotopy is 0.
\section{The group $O(2)$}
\label{sec:o2}

In Sections \ref{sec:o2}, \ref{sec:o2cyclic} and \ref{sec:o2dihedral},
the ambient group is $G=O(2)$. In the present section we
introduce the subgroups and representations, whilst in Sections
\ref{sec:o2cyclic} and \ref{sec:o2dihedral} we describe the two
summands of $[S^0, S^U]^G_*$.
All the calculations in this section can be made with \cite{o2q}, since 
the Adams spectral sequence collapses to a short exact sequence.  

\subsection{Subgroups in blocks}

We describe the space $\fX_G=\sub(G)/G$ of subgroups of $G=O(2)$ up to
conjugacy with the Hausdorff metric topology. First, we write
$$\cC=\{ C_1, C_2, \ldots \} \mbox{ and }
\mcD=\{ (D_2), (D_4), \ldots \} $$
for the cyclic and dihedral parts. The subgroup $SG=SO(2)$ is the 
limit point of $\cC$,  and the subgroup $G=O(2)$ is the 
limit point of $\mcD$. Thus, with the Hausdorff metric topology we
have a partition into clopen subsets
$$\fX_G=\cCh\amalg \mcDh.$$
We refer to $\cCh$ as the 
{\em cyclic block} (it is `dominated' by $SG$) and to $\mcDh$ as the 
{\em dihedral block} (it is `dominated' by $G$). 

It will not play an explicit role here, but in fact the space $\fX_G$ also
has a Zariski topology. The displayed partition is also a disjoint 
union of clopen sets in the Zariski
topology. In the Zariski topology, the subspace $\mcDh$ is 
still the one point compactification of $\mcD$, but in the Zariski
topology $\cCh$ is homeomorphic to $\spec (\Z)$, so that  
the closure of the point $SG$ is the whole space.

\subsection{Models}

The standard model $\cA(G)$ is an abelian category  which provides
 an algebraic model for rational $G$-spectra in the sense that there
 is a Quillen equivalence \cite{BarnesO(2)}. 
  $$  \mbox{$G$-spectra}\simeq \mbox{DG-$\cA(G)$} .$$
  Since all objects are formal we work with $\cA(G)$, but use the notation imported
  from $G$-spectra.

Since the block decomposition is as clopen sets, the algebraic model 
of rational $O(2)$ spectra splits in a corresponding fashion 
$$\cA (G|\cCh) \times \cA (G|\mcDh). $$
The cyclic part has already been essentially dealt with in 
Section \ref{sec:so2} since 
$$\cA (G|\cCh)=\cA(SG)[W]; $$
where $W=W_G(SG)$ is of order 2, and acts on the coefficient rings
$\Q[c]$ to negate $c$. The dihedral part, $\cA (G|\mcDh)$, 
is simply a category of equivariant
sheaves. More precisely, we associate the Weyl group $W_G(H)$ to each
point $H$ of $\mcD^{\hash}$, and write $\cW$ for this collection of groups (it is a
group of order 2 for each dihedral subgroup and the trivial group for
$O(2)$). In an equivariant sheaf, $W_G(H)$ acts on the stalk over $H$,
and the germ-spreading map is equivariant: we then have
$$\cA(G|\mcDh)=\cW-\shv/\mcD^{\hash}. $$

\subsection{Representations}
It is reassuring to  name explicit simple representations.
We write $\sigma_1$ for the natural representation of $O(2)$ on
$\R^2$, and we write $\sigma_n$ for the pullback of $\sigma_1$ along
the projection $O(2)\lra O(2)/C_n\cong O(2)$ (where the isomorphism is
chosen to preserve orientation).  We write $\delta$ for the inflation of
the sign representation of $O(2)/SO(2)$ on $\R$. The representations,
$1, \delta, \sigma_1, \sigma_2, \ldots$ give a complete list of
representatives for simple representations over $\R$.
Thus
$$U=a+b\delta +\sum_{n\geq 1} c_n\sigma_n$$
for integers $a,b, c_n$,  almost all zero. 
Associated to this we have the fixed point dimension functions
$$d_U(s) =\dim_{\R}(U^{C_s})=a+b+2\sum_{s|n}c_n$$
on cyclic groups, and 
$$d'_U(t) =\dim_{\R}(U^{D_{2t}})=a+\sum_{t|n}c_n$$
on dihedral groups. 

Note in particular that if $U^G=0$, any even dimensional
representation contains an even number of copies of $\delta$. 

If we only care about virtual representations, we have
$$RO(O(2))=\Z[\sigma_1, \delta]/(\delta^2=1, \sigma_1 \delta=\sigma_1).$$

\section{The cyclic block in $O(2)$}
\label{sec:o2cyclic}
\subsection{The role of $\delta$}
We pause to emphasize that the representation $\delta$ of $G$ restricts to
the trivial representation of $SG$, and write $b(U)$ for the
multiplicity of $\delta$ in $U$.
If $U^G=0$ we have $U^{SG}=b(U)$. We will need one fact. 

\begin{lemma}
  \label{lem:Waction}
If $U^{G}=0$ then the action of $W$ on $H_*^{SG}(S^U)$ on each degree
depends only on the parity of $b(U)$.
\end{lemma}

\begin{proof}
 We argue by showing the the maps $S^A\lra S^{A\oplus B}$ induce
 isomorphisms in $SG$-equivariant Borel cohomology in degrees where
 both are non-zero. These isomorphisms are $W$-equivariant.

 When adding a representation $B=\sigma_n$ we use the cofibre sequence
 $S(\sigma_n)_+\lra S^0 \lra S^{\sigma_n}$. The point is that
 $S(\sigma_n)=SG/C_n$ so that  $H^{SG}_*(S(\sigma_n)_+)$ is in degree
 0. It follows that  $H_{2k}^{SG}(S^0)$ is unchanged except in degree
0. Iterating this, we see that if $U$ has no summands $\delta$, the homology
$H_*^{SG}(S^U)$ is the same as $H_*^{SG}(S^0)$ in each degree in which
it is non-zero.

When adding a representation $\delta$ we use the cofibre sequence
$W_+\lra S^0\lra S^{\delta}$. This induces the short 
exact sequence 
$$0\lra H^{SG}_*(S^{\delta -1})\lra H^{SG}_*(S^0)[W]\lra
H^{SG}_*(S^0)\lra 0.$$
  In each degree this is either
$0\lra \Q \lra \Q[W]\lra \Qt\lra 0$ or $0\lra \Qt \lra \Q[W]\lra
\Q\lra 0$ so that suspending by $\delta$ shifts  degree by 1 and tensors with $\Qt$. 
   \end{proof}

\subsection{The calculation}
If $U$ is a representation of $G=O(2)$ with $U^G=0$ then the cyclic
contribution to $[S^0, S^U]^G$ is
$$[S^0, S^U]^G_{\cC}=([S^0,S^U]^{SG})^W.$$

We separate into two cases according to the parity of $b(U)$. One
could design a notation to capture it all at once, but  the
cost in readability is not worth it.

\subsubsection{$b(U)$ even}
In this case $\delta$ only shifts the answer, so if $U=U'\oplus 
b(U)\delta$ we have 
$$[S^0,S^U]^G_{\cC}=\Sigma^{b(U)}[S^0,S^{U'}]^G_{\cC}.$$
Now we can apply Section \ref{sec:so2} directly since $(U')^{SG}=0$. 

If $d_{U'}\geq 0$ then 
$$[S^0,S^{U'}]^{SG}_{ev}=[S^0,S^{U'}]^{SG}_0=\Q \mbox{ and $W$ acts 
  trivially. }$$
If $d_{U'}$ takes a negative value the even part is zero.
All other homotopy is in odd degrees, where it is 
$$[S^0,S^{U'}]^{SG}=\Sigma^{U'-1}\II/\Delta_{-1}, $$
where $\Delta_{-1}$ is the diagonal copy of $\Q$ in degree $-1$ if 
$d_{U'}$ takes a negative value, and is zero otherwise. 

We have $\Sigma^{U'}\II=\bigoplus_s \Sigma^{d_{U'}(s)}I_s$. 
By Lemma \ref{lem:Waction}, the action of $W$ on the nonzero terms 
$(\Sigma^{d_{U'}(s)}I_s)_{2n}$ depends only 
on $n$ (and not on $s$ or $U'$). Indeed, $W$ acts as $-1$ on $c$ so
that $\Lc$ has $\Q$ in degrees 0 mod 4 and $\Qt$ in degrees 2
mod 4, and we have
$$\Lc^W=\Q[d,d^{-1}] \mbox{ where } d=c^2 \mbox{ is of degree }
-4.$$
The representation $\Sigma^{d_{U'}(s)}I_s$ is the appropriate quotient of $\Lc$,
and its $W$-invariants are the corresponding quotient of $\Q[d,
d^{-1}]$, namely the part in degrees 0 mod 4, which is a suspension of $\Qd^*$. There is a 
splitting of cases depending on whether the bottom class is invariant
or not: $(\Qc^*)^W=\Qd$ whilst $(\Qc^*\tensor
\Qt)^W=\Sigma^2\Qd^*$. We also note that $d_U(s)=b(U)+d_{U'}(s)$.

\begin{cor}
If $b(U)$ is even then in even degrees $[S^0,S^U]^{G}_*$ is 
$\Sigma^{b(U)}\Q$ if $d_{U'}\geq 0$ and 
it is 0 if  $d_{U'}$ takes a negative value. 

In odd degrees $[S^0,S^U]^{G}_*$ is 
$\bigoplus_s(\Sigma^{d_{U}(s)-1} I_s)^W$ and it is the quotient of this by the diagonal 
$\Delta_{b(U)-1} \cong \Q$ in degree $b(U)-1$ if 
$d_{U'}$ takes a negative value. If $d_s(U')$ is 0 mod 4 then  
$(\Sigma^{d_{U}(s)-1}I_s)^W
=\Sigma^{3+d_{U}(s)}\Qd^* $, and  if 
$d_s(U')$ is 2 mod 4 then $(\Sigma^{d_{U}(s)-1}I_s)^W=\Sigma^{1+d_{U}(s)}\Qd^* $. 
\end{cor}
  
Thus if $b(U)$ is even, even degrees are just like the $SG$ case 
(i.e., 0 or 1 dimensional in degree $b(U)$ depending on whether $d_{U'}$
takes a negative value). However in odd degrees there are two parts,
in degrees $b(U)+ 1$ 
mod 4 and $b(U)+3$ mod 4. In  degrees $b(U)+3$ mod 4 and $\geq
b(U)+3$, the dimension is
countably infinite, and in degrees $b(U)+1$ mod 4,  the dimension is $|\{
s\st d_{U'}(s)=2 \mbox{ mod } 4\}|$. In  negative odd degrees $<
m_{U}+1$ the answer is zero.

\subsubsection{$b(V)$ odd}
In this case $\delta$ both shifts and twists the answer, so if $U=U'\oplus 
b(U)\delta$ we have 
$$[S^0,S^U]^G_{\cC}=\Sigma^{b(U)}([S^0,S^{U'}]^{SG}\tensor \Qt)^W.$$
Once again we apply Section \ref{sec:so2}, and the work in the
previous subsection lets us write down the answer.

\begin{cor}
  If $b(U)$ is odd then $[S^0,S^U]^{G}_*=0$ in odd degrees.

  In even degrees $[S^0,S^U]^{G}_*$ is 
  $\bigoplus_s (\Sigma^{d_{U}(s)-1}I_s\tensor \Qt)^W$.
If $d_s(U')$ is 0 mod 4 then
    $(\Sigma^{d_{U}(s)-1}I_s\tensor \Qt)^W=    \Sigma^{1+d_{U}(s)}\Qd^* $, and  if 
    $d_s(U')$ is 2 mod 4 then $(\Sigma^{d_{U}(s)-1}I_s\tensor \Qt)^W
    =\Sigma^{3+d_{U}(s)}\Qd^* $. 
  \end{cor}

Thus if $b(U)$ is odd, the homotopy is zero in odd degrees (i.e., in
degrees congruent to $b(U)$ mod 2). 
 In even degrees there are two parts, in degrees $b(U)+1$ 
mod 4 and $b(U)+3$ mod 4. In degrees $b(U)+1$ mod 4 and $\geq b(U)+1$, the dimension
is countably infinite, and in degrees $b(U)+3$ mod 4 it is
dimension is $|\{
s\st d_{U'}(s)=2 \mbox{ mod } 4\}|$. In  negative odd degrees $<
m_{U}+1$ the answer is zero.

 \section{The dihedral block in $O(2)$}
  \label{sec:o2dihedral}
On the dihedral block, $S^0$ corresponds to the constant sheaf
$\Q_{\mcDh}$ over $\mcDh$, and $\Sigma^US^0$ corresponds to the sheaf we may call
$\Sigma^U\Q_{\mcDh}$.

To explain, $\Sigma^U\Q_{\mcDh}$  is the same as $\Q_{\mcDh}$
except at the finitely many stalks where $d'_U(t)\neq 0$. At the $t$th
stalk, we have $\Sigma^{d'_U(t)}\Q^{\pm}$. In all cases the action of
$W_t=W_G(D_{2t})=D_{4t}/D_{2t}$ is the action on $H_{d_U(t)}(S^{U^{D_{2t}}})$.

It follows that
$$[S^0,S^U]^G_{\mcDh}=\Hom_W(\Q_{\mcDh},
\Sigma^U \Q_{\mcDh})=\Gamma(\Sigma^{U}\Q_{\mcDh})^W$$
(i.e., equivariant sections). 

It remains to identify the actions. Of course $\delta^D=0$ for all
dihedral subgroups $D$, so we can ignore the $\delta$ summands.

\begin{lemma}
The representation $\sigma_n^{D_{2t}}$ of 
$W_t=W_G(D_{2t})=D_{4t}/D_{2t}$ is the trivial representation $\Q$ if $2t|n$, is the sign 
representation $\Qt$  if $t|n$ and $2t\not |n$ and is zero otherwise. 
\end{lemma}

\begin{proof}
We see that  $\sigma_n^D$ is either 0 or 1 dimensional, and in fact 
$\dim_{\R}(\sigma_n)^{D_{2t}}=1$ precisely if $t|n$. 
\end{proof}

\begin{cor}
The group $[S^0, S^U]^G_{\mcD}$ is almost equal to $\Gamma
(\Q_{\mcDh})=C(\mcDh, \Q)$ in degree 0. The difference occurs in the
finitely many spots at which $d'_U(t)\neq 0$. If $d'_U(t)\neq 0$, a
factor $\Q$ is removed from degree 0 and replaced by a term in degree
$d'_U(t)$. The new term is $\Q$ if $\Sigma_n c_n$ is even, where the
sum is over those $n$ so that  $t|n$, $2t\not |n$. Otherwise the new
term is 0.
  \end{cor}

The main thing to observe about the dihedral block is that the
contribution to $[S^0, S^U]^G_*$ is in a finite range of degrees, and
is finite dimensional except in degree 0 ($=\dim U^G$), where it is of
countably infinite dimension. 
  \section{The group $SO(3)$}
\label{sec:so3}
  In this section the ambient group is $\Gamma=SO(3)$. It is a matter
  of elementary representation theory to deduce the answer from those
  for $G=O(2)$.
All the calculations in this section can be made with \cite{so3q}, since 
the Adams spectral sequence collapses to a short exact sequence.  
  
  \subsection{Subgroups in blocks}
   This group has 7 blocks of subgroups. The cyclic block $\cC$
  is exactly as before, whilst the dihedral one
  $\mcD'=\{ D_6, D_8, D_{10}, \ldots \}$ omits $D_2$
  and $D_4$ ($D_2$ because it is conjugate to $C_2$ in
  $SO(3)$ and $D_4$ because its Weyl group in $SO(3)$ is
  $\Sigma_3$). Otherwise these blocks are the same as for $O(2)$. Then
  there are 5 isolated blocks $\{ (H)\}$, which we record in
  Subsection \ref{subsec:isolated} with their  Weyl group,  as  $(H,W_G(H))$.

\subsection{Models}

The standard model $\cA(\Gamma )$ is an abelian category  which provides
 an algebraic model for rational $\Gamma$-spectra in the sense that there
 is a Quillen equivalence \cite{KedziorekSO(3)}: 
  $$  \mbox{$\Gamma$-spectra}\simeq \mbox{DG-$\cA(\Gamma)$} .$$
  Since all objects are formal we'll work with $\cA(\Gamma)$, but use
  the notation 
imported   from $\Gamma$-spectra.

Since the block decomposition is as Zariski clopen sets, the algebraic model 
of rational $SO(3)$ spectra splits as a product of 7 categories 
in a corresponding fashion 
$$\cA(\Gamma)\simeq \cA (\Gamma |\cCh) \times \cA (\Gamma |\mcDph)\times 
\prod_H \Q[W_{\Gamma}(H)]\modules. $$
The cyclic part has already been essentially dealt with in 
Section \ref{sec:so2} (or rather Section \ref{sec:o2cyclic}) since
$\cA (\Gamma |\cCh)$ is very close to $\cA(SG)[W]$
and the dihedral part is simply a category of equivariant sheaves 
$$\cA(\Gamma|\mcDph)=\cW-\shv/\mcDph. $$

The difference between $\cA (\Gamma|\cCh)$ and $\cA(SG)[W]$ is purely
at the trivial subgroup where the Weyl group $W_{\Gamma}(1)= \Gamma$
has cohomology ring $\Q[d]$, whereas  in $O(2)$ the Weyl group
$W_G(1)=G$ has  governing ring $\Q[c][W]$. 
Thus $\cA (\Gamma|\cCh)$ is the retract specified by
replacing the $\Q[c][W]$-module  $e_1N$  by the $\Q[c]$-module of
 $W$-fixed points. In effect, on cyclic parts, 
 the $\Gamma$ spectra correspond to $G$-spectra where the free part is
 extended from the $W$-invariants.

 \subsection{Representations}
The other change is in the representation theory. Now the simple 
modules over $\R$ are $W_1, W_3, W_5, \ldots $ one of each odd 
dimension (see for example \cite{BtD}, especially II.5).

The complexification of $W_{2i+1} $ has weights $z^{2i}, z^{2i-2},
\ldots, z^0, z^{-2}, \ldots , z^{-2i}$ with multiplicity
1 as a representation of $SU(2)$. Accordingly, if we write 
$\delta(i,s)=|\{ j\in [1, i]\st s|j\}| $, we have
$$d_{W_{2i+1}}(s)=1+2\delta (i,s), \mbox{ and }
d'_{W_{2i+1}}(t)=\delta (i,t). $$
Thus if $U=\sum_ia_iW_{2i+1}$ we have
$$d_U(s)=\sum_i(1+2\delta (i,s))a_i \mbox{ and }
d'_U(t)=\sum_i\delta (i,t)a_i .$$

By considering weights and the overall determinant we see that 
$$\res^{SO(3)}_{O(2)}W_{2i+1}=\delta^i +\sigma_1 +\cdots +\sigma_i.$$

  \subsection{Cyclic block}
For the cyclic block, despite the intricacies at the trivial subgroup,
restriction to $O(2)$ gives an isomorphism for spheres.

\begin{lemma}
For a virtual representation $U$ of $\Gamma$, the restriction map 
$$[S^0, S^U]^{SO(3)}_{\cC}\stackrel{\cong}\lra 
[S^0,  S^U]^{O(2)}_{\cC}=([S^0,  S^U]^{SO(2)})^W  $$
is an isomorphism. 
\end{lemma}

\begin{proof}
Consider  the resolution  
$$0\lra S^U\lra e(\Q)\lra f(\Sigma^U \II)\lra 0$$ 
in $\cA (G|\cCh)$, which we note is $W$-equivariant. 
To move this to $\cA (\Gamma|\cCh)$, the only change is that $\Sigma^UI_1$ is 
replaced by its $W$-fixed points. However, to calculate homotopy 
groups we take invariants in any case. Accordingly, the cochain complex 
$$\Q  \lra (\Sigma^U \II)^W$$
giving $[S^0,S^U]^G_{\cC}$ is exactly the same as that giving $[S^0,S^U]^{\Gamma}_{\cC}$. 
\end{proof}

    \subsection{Dihedral block}
  For the dihedral block, the forgetful map gives an isomorphism 
  $$[S^0, S^U]^{SO(3)}_{\mcDp}\stackrel{\cong}\lra [S^0, S^U]^{O(2)}_{\mcDp}.$$
The only thing to highlight is that the idempotent in $O(2)$ is that
corresponding to $\mcDp=\mcD\setminus \{ D_2, D_4\}$, so we have
removed the two summands corresponding to $D_2$ and $D_4$.

\subsection{Isolated blocks}
\label{subsec:isolated}
  Here the model is elementary 
  $$\cA (\Gamma |(H))=\Q W_{\Gamma}(H)\modules . $$
  The 5 isolated blocks (in the form $(H, W_G(H))$) are 
  $$(SO(3),1), (A_5, 1), (\Sigma_4, 1), (A_4, C_2), (D_4, \Sigma_3).$$
The contributions are
$$[S^0, S^U]^{\Gamma}_{(H,W_{\Gamma}(H))}=\Hom_{W_{\Gamma}(H)}(\Q ,
H_*(S^{U^H}))=H_*(S^{U^H})^{W_{\Gamma}(H)}. $$
  It remains only to determine the dimensions of the fixed point
  representations $U^H$, and in the case that the Weyl group is
  non-trivial, to  ask how it acts. In effect this is the group $\Sigma_4/D_4$ acting on
  $W_{2i+1}^{D_4}$, and $\Sigma_4/A_4$ acting on
  $W_{2i+1}^{A_4}=(W_{2i+1}^{D_4})^{C_3}$. It is routine to work out
  the exact answer using the well known 
relationship between the representations of $SU(2)$ and $SO(3)$  as 
summarised in Section \ref{sec:su2}.  First, we may use characters of complex
  representations since fixed points commute with complexification. 
The restriction of  the natural representation  of $SU(2)$ is easily
spotted from the character table of the binary octahedral group. This
gives the characters of its symmetric powers which gives the required
information about $W_{2i+1}$. 

  \section{The group $Pin (2)$}
\label{sec:pin2}
  In this section the ambient group is $\Gt=Pin(2)$.
The algebraic model of $\Gt=Pin(2)$ and of $G=O(2)$ are essentially the same
in the sense that $Pin(2)$ has two blocks essentially the same as that
for $O(2)$. However, some care is necessary.

All the calculations in this section are straightforward adaptions from \cite{o2q}, since 
the Adams spectral sequence collapses to a short exact sequence. 

\subsection{Subgroups in blocks}
First, we let
$\cCt =\{C_1, C_2, \ldots
\}$ and $\cQ=\{Q_4, Q_8, Q_{12}, \ldots \}$, where $Q_{4n}$ is the
quaternion group of order $4n$. There is again a cyclic
block $\cCth$, and it is isomorphic to $\cCh$, and the Weyl groups and
sheaf of rings are also the same. However the isomorphism $\cCt\lra
\cC$ takes each group to the group of the same order, and this is not
induced by the quotient map $p: Pin(2)\lra O(2)$. On the other hand the
quotient map does induce the isomorphism $\cQ\lra \mcD$, since we have
$Q_{4n}/C_2=D_{2n}$ for the indvidual groups.

\subsection{Models}
The standard model $\cA(\Gt)$ is an abelian category  which provides
 an algebraic model for rational $\Gt$-spectra in the sense that there
 is a Quillen equivalence \cite{gq1}:
  $$  \mbox{$\Gt$-spectra}\simeq \mbox{DG-$\cA(\Gt)$} . $$
  Since all objects are formal we'll work with $\cA(\Gt)$, but use the notation familiar
  from $\Gt$-spectra.

Since the block decomposition is as clopen sets, the algebraic model 
of rational $Pin(2)$ spectra splits in a corresponding fashion 
$$\cA(\Gt)=\cA (\Gt|\cCth) \times \cA (\Gt|\cQh). $$
The cyclic part has already been essentially dealt with in 
Section \ref{sec:so2} since 
$$\cA (\Gt|\cCh)=\cA(S\Gt)[W], $$
and the dihedral part is simply a category of equivariant sheaves 
$$\cA(\Gt|\cQh)=\cW-\shv/\cQh. $$

\subsection{Representations}
The other difference is in the representation theory. The simple
representations of $Pin (2)$ include the inflations from $O(2)$, which
we continue to call $1, \sigma_1, \sigma_2, \ldots, \delta$. These
have kernels $Pin(2), C_2, C_4, \ldots, Spin(2)$.  However
there is also the representation $h_1$ of $Pin(2)\subseteq SU(2)\cong
Sp(1) $ on the quaternions, which is of real dimension 4. Similarly, for
each cyclic group of odd order $m$ we may pull back $h_1$ along
$Pin(2)\lra Pin(2)/C_m$ to give a represention $h_m$. The
representations $h_1, h_3, h_5, \ldots $ have kernels $C_1, C_3, C_5,
\ldots $.

The action on $\sigma_n$ is via the quotient, so that
$(\sigma_n)^K=\sigma_n^{pK}$. Similarly, $h_m^K=h_1^{p_mK}$, so that
$d_{h_m}(s)=4$ if $s|m$ and is zero otherwise, and $d'_{h_m}(t)=0$. 

Thus if 
$$U=a+b\delta +\sum_nc_n \sigma_n+\sum_md_mh_m $$
we have dimension functions $d_U: \cCt\lra \N$ and $d'_U: \cQ\lra \N$
given by 
$$d_U(s)=a+b +2\sum_{s\; odd, s|m}c_n +2\sum_{s \;even, s/2|m}c_n+
4\sum_{s|m} d_m,  d'_U(t)=a+\sum_{ t|m}c_n . $$

\section{The group $SU(2)$}
\label{sec:su2}
In this section the ambient group is $\Gammat=SU(2)$.

All the calculations in this section are straightforward adaptions from \cite{so3q}, since 
the Adams spectral sequence collapses to a short exact sequence.  

\subsection{Subgroups in blocks}
The group $\Gamma$ again has 7 blocks. 

Again we have the cyclic block $\cCth$, and the quaternion block
$\cQph$, with the same relationship to the blocks in $Pin(2)$ as those
in $SO(3)$ had to $O(2)$. There are then 5 singleton blocks
corresponding to the double covers $\Ht$ of the groups $H$ in
$SO(3)$. These isolated blocks are isomorphic to the ones in $SO(3)$,
since in the block $(\Ht,W_{\Gt}(\Ht))$ we have 
$W_{\Gt}(\Ht)=W_G(H)$. 

\subsection{Models}
The standard model $\cA(\Gammat)$ is an abelian category  which provides
 an algebraic model for rational $\Gammat$-spectra in the sense that there
 is a Quillen equivalence \cite{gq1}
  $$  \mbox{$\Gammat$-spectra}\simeq \mbox{DG-$\cA(\Gammat)$} . $$
  Since all objects are formal we'll work with $\cA(\Gammat)$, but use the notation familiar
  from $\Gammat$-spectra.

Since the block decomposition is as clopen sets, the algebraic model 
of rational $SU(2)$-spectra splits in a corresponding fashion 
$$\cA(\Gammat)\simeq \cA (\Gammat |\cCh) \times \cA (\Gammat |\cQph)\times 
\prod_H \Q[W_{\Gamma}(H)]\modules. $$
The cyclic part has already been essentially dealt with in 
Section \ref{sec:so2} since 
$\cA (\Gammat |\cCth)$ is very close to $\cA(S\Gt)[W],$
with the adaptions at the trivial group as in Section \ref{sec:so3}. 
The dihedral part is simply a category of equivariant sheaves 
$$\cA(\Gammat|\cQph)=\cW-\shv/\cQph. $$
\subsection{Representations}
The representations of $SU(2)$ over $\C$ are the symmetric powers of the
natural representation $V_1, V_2, V_3, \ldots $;  $V_i$ of dimension
$i$ consists of homogeneous polynomials of degree $i-1$ (see\cite{BtD} for example). The
representation $V_{2i}$ is complex so as a representation 
over $\R$ it is  of dimension $4i$. The representations $V_{2i+1}$
are complexifications of the modules $W_{2i+1}$ of real dimension $2i+1$.

For the fixed point degrees, we note that for $C_s\subseteq SU(2)$, 
$$W_i^{C_s}=W_i^{pC_s}, W_i^{Q_{4t}}=W_i^{D_{2t}}$$
Similarly, 
$$\dim_{\R}(V_{2i}^{C_s})=|\{ j\in [1,i]\st s|j\}|, V_{2i}^{Q_{4t}}=0.$$
 \subsection{Isolated blocks}
  Here the model is elementary. Noting that
  $W_{\Gammat}(\Ht)=N_{\Gammat}(\Ht)/\Ht\cong
  N_{\Gamma}(H)/H=W_{\Gamma}(H)$, and noting that the fixed point sets
  of the representations $V_{2i}$ are zero, the answers can be read
  off those in Subsection \ref{subsec:isolated}.


\begin{thebibliography}{1}

\bibitem{BarnesO(2)}
David Barnes.
\newblock Rational {$O(2)$}-equivariant spectra.
\newblock {\em Homology Homotopy Appl.}, 19(1):225--252, 2017.

\bibitem{BtD}
T. Br\"ocker and T. tom Dieck 
\newblock Representations of compact Lie groups
\newblock {\em Graduate texts in Mathematics}, Springer-Verlag 1985

\bibitem{o2q}
J.~P.~C. Greenlees.
\newblock Rational {O}(2)-equivariant cohomology theories.
\newblock In {\em Stable and unstable homotopy (Toronto, ON, 1996)}, volume~19
  of {\em Fields Inst. Commun.}, pages 103--110. Amer. Math. Soc., Providence,
  RI, 1998.

\bibitem{s1q}
J.~P.~C. Greenlees.
\newblock Rational {$S\sp 1$}-equivariant stable homotopy theory.
\newblock {\em Mem. Amer. Math. Soc.}, 138(661):xii+289, 1999.

\bibitem{so3q}
J.~P.~C. Greenlees.
\newblock Rational {$\rm SO(3)$}-equivariant cohomology theories.
\newblock In {\em Homotopy methods in algebraic topology ({B}oulder, {CO},
  1999)}, volume 271 of {\em Contemp. Math.}, pages 99--125. Amer. Math. Soc.,
  Providence, RI, 2001.

\bibitem{AGconj}
J.~P.~C. Greenlees.
\newblock Triangulated categories of rational equivariant cohomology theories.
\newblock {\em Oberwolfach Reports, pages 480–488, 2006. (cit. on p. 2)},
  2006.

\bibitem{gq1}
J.~P.~C. Greenlees.
\newblock Algebraic models for one-dimensional categories of rational
  {$G$}-spectra.
\newblock {\em Preprint, 28pp, arXiv:2501.11200}, 2025.

\bibitem{GSerrefinite}
J.~P.~C. Greenlees and J.~D. Quigley.
\newblock Ranks of {$RO(G)$}-graded stable homotopy groups of spheres for
  finite groups {$G$}.
\newblock {\em Proc. Amer. Math. Soc. Ser. B}, 10:101--113, 2023.

\bibitem{tnqcore}
J.~P.~C. Greenlees and B.~Shipley.
\newblock An algebraic model for rational torus-equivariant spectra.
\newblock {\em J. Topol.}, 11(3):666--719, 2018.

\bibitem{KedziorekSO(3)}
Magdalena K{\c{e}}dziorek.
\newblock An algebraic model for rational {${\rm SO}(3)$}-spectra.
\newblock {\em Algebr. Geom. Topol.}, 17(5):3095--3136, 2017.

\end{thebibliography}
\end{document}